\documentclass[a4paper,english,fontsize=10pt,parskip=half,abstracton]{scrartcl}
\usepackage{babel}
\usepackage[utf8]{inputenc}
\usepackage[T1]{fontenc}
\usepackage[a4paper,left=20mm,right=20mm,top=30mm,bottom=30mm]{geometry}
\usepackage{amsmath}
\usepackage{amsthm}
\usepackage{amssymb}
\usepackage{enumerate}
\usepackage{aliascnt}
\usepackage[bookmarks=false,
            pdftitle={Regular orbits of coprime linear groups in large characteristic},
            pdfauthor={Benjamin Sambale},
            pdfkeywords={},
            pdfstartview={FitH}]{hyperref}
\usepackage{tikz}
\usetikzlibrary{matrix,arrows,positioning,shapes.geometric,decorations.pathreplacing}

\newtheorem{Thm}{Theorem} %[section]
\newaliascnt{Lem}{Thm}

\aliascntresetthe{Lem}
\newaliascnt{Prop}{Thm}
\newtheorem{Prop}[Prop]{Proposition}
\aliascntresetthe{Prop}

\numberwithin{equation}{section}

\allowdisplaybreaks[1]

\renewcommand{\phi}{\varphi}
\newcommand{\C}{\operatorname{C}}
\newcommand{\N}{\operatorname{N}}
\newcommand{\Z}{\operatorname{Z}}

\newcommand{\pcore}{\operatorname{O}}
\newcommand{\GL}{\operatorname{GL}}
\newcommand{\SL}{\operatorname{SL}}

\mathchardef\ordinarycolon\mathcode`\:  %defines a nice ":=" 
 \mathcode`\:=\string"8000
 \begingroup \catcode`\:=\active
   \gdef:{\mathrel{\mathop\ordinarycolon}}
 \endgroup

\title{Regular orbits of coprime linear groups\\ in large characteristic}
\author{Benjamin Sambale\footnote{Fachbereich Mathematik, TU Kaiserslautern, 67653 Kaiserslautern, Germany, \href{mailto:sambale@mathematik.uni-kl.de}{sambale@mathematik.uni-kl.de}}}
\date{\today}

\begin{document}
\frenchspacing
\maketitle
\begin{abstract}\noindent
We prove that a finite coprime linear group $G$ in characteristic $p\ge\frac{1}{2}(|G|-1)$ has a regular orbit. This bound on $p$ is best possible. We also give an application to blocks with abelian defect groups. 
\end{abstract}


\section{Introduction}
Linear groups, i.\,e. subgroups of $\GL(V)$ where $V$ is a vector space, play an important role in many branches of group representation theory. In the present note we are interested in the situation where $V$ is an elementary abelian $p$-group and $G$ is a $p'$-group. Then by Maschke's Theorem, the action of $G$ on $V$ is semisimple. 
A crucial fact used in the solution of the so-called $k(GV)$-problem~\cite{kGVproblem} is that $G$ often has a \emph{regular orbit} on $V$, that is, an orbit of size $|G|$. This is well-known for abelian groups $G$ and many other special cases have been handled in the literature (see for instance \cite{Espuelas,Fleischmann,Hargraves,KoehlerPahlings,RegOrbits}).
It is also known that regular orbits exist if $p$ is “large”, for example if $p>|G|$ (see \cite[Lemma~2.2]{Robinson}). 
The main result of this paper establishes the existence of a regular orbit under the weaker condition $p\ge\frac{1}{2}(|G|-1)$. The proof relies on a classification of finite groups with “many” minimal subgroups given by Burness-Scott~\cite{BurnessStuart}.
On the other hand, for every odd prime $p$ we construct linear groups $G$ without regular orbits such that $p=\frac{1}{2}|G|-1$. 

These results are motivated by Brauer's $k(B)$-Conjecture which is still open even for blocks with abelian defect groups. In this situation, $V$ is a defect group of a block $B$ of a finite group and $G$ is the corresponding inertial quotient. The number of irreducible characters in $B$ is denoted by $k(B)$. 
By a result of Robinson~\cite{Robinson}, Brauer's $k(B)$-Conjecture, i.\,e. $k(B)\le|V|$, holds provided $G$ has a regular orbit on $V$.
In previous papers~\cite{Sbrauerfeit,SambaleC4,SambaleRank3} we have applied some of the techniques from the $k(GV)$-problem to this more general situation. Now in the present paper we will give new application of our main theorem.

\section{Main theorem}

In the following we denote a cyclic group of order $n$ by $C_n$. A dihedral group of order $n$ is denoted by $D_n$ and the symmetric and alternating groups of degree $n$ are $S_n$ and $A_n$ respectively. Moreover, for a finite group $G$ we set $G^n:=G\times\ldots\times G$ ($n$ copies). The exponent of $G$ is defined by $\exp(G):=\min\{n\ge 1:g^n=1\ \forall g\in G\}$. 
Note that the minimal subgroups of a finite group are precisely the subgroups of prime order.

\begin{Prop}[{\cite[Theorem~1.1]{BurnessStuart}}]\label{prop}
If a non-trivial finite group $G$ has more than $\frac{1}{2}|G|-1$ minimal subgroups, then one of the following holds:
\begin{enumerate}[(i)]
\item\label{one} $G\cong A\rtimes C_2$ where $A$ is abelian and $C_2$ acts as inversion.
\item $G\cong C_2^n\rtimes C_2$ for some $n\ge 2$.
\item $G\cong C_2^{2n}\rtimes C_3$ for some $n\ge 1$ (Frobenius group).
\item $G\cong D_8*\ldots*D_8\times C_2^n$ for some $n\ge 0$ (central product).
\item $\exp(G)=3$.
\item $G\cong D_8^2\times C_2^n$ for some $n\ge 0$.
\item $G\cong S_3\times D_8\times C_2^n$ for some $n\ge 0$.
\item $G\in\{S_3^2,\, S_4,\, A_5\}$.
\end{enumerate}
\end{Prop}

Note that \eqref{one} in \autoref{prop} includes all non-trivial elementary abelian $2$-groups. Now in our main theorem we prove slightly more than what was promised in the introduction.

\begin{Thm}\label{main}
Let $G$ be a $p'$-automorphism group on a finite $p$-group $P$ such that $|G|\le 2p+9$. Then one of the following holds:
\begin{enumerate}[(i)]
\item $G$ has a regular orbit on $P$.
\item $G\cong D_{2p+2}$.
\item $G\cong D_8*C_4$ and $p=5$.
\end{enumerate}
\end{Thm}
\begin{proof}
Let $G$ be a minimal counterexample. By a result of Hartley-Turull~\cite[Lemma~2.6.2]{HartleyTurull}, we may assume that $P$ is elementary abelian. Then $G$ is non-abelian (see for example \cite[Corollary~4.I]{Fleischmann}).
If $\overline{G}$ is a proper quotient of $G$, then either $|\overline{G}|\le\frac{1}{3}|G|\le\frac{2}{3}p+3<2p+2$ or $p>2$ and $|\overline{G}|=\frac{1}{2}|G|\le p+\frac{9}{2}<2p+2$. In particular, every proper quotient of $G$ has a regular orbit. By \cite[Lemmas~2.I and 3.I]{Fleischmann} it suffices to assume that $P$ is an absolutely irreducible $G$-module over a finite field with $q=p^t$ elements. In particular, the center $\Z(G)$ of $G$ is cyclic. 
Let $\mathcal{M}$ be the set of minimal subgroups of $G$. Then for every $x\in P$ there exists $M\in\mathcal{M}$ such that $M\le\C_G(x)$. Hence,
\begin{equation}\label{union}
P=\bigcup_{M\in\mathcal{M}}\C_P(M).
\end{equation}
Since $G$ acts faithfully, $\C_P(M)<P$ for every $M\in\mathcal{M}$. This shows $m:=|\mathcal{M}|\ge 2$ and
\begin{equation}\label{sum}
|P|<\sum_{M\in\mathcal{M}}{\lvert\C_P(M)\rvert}\le m\frac{|P|}{q}.
\end{equation}
Therefore, $\frac{1}{2}(|G|-9)\le p\le q\le m-1$. 
Let $|G|=2p+k$ with $k\le 9$ (here $k$ may be negative). We discuss four cases.

\textbf{Case 1:} $k\in\{3,5,7,9\}$.\\
Here $G$ has odd order. Since two distinct minimal subgroups of $G$ intersect trivially, it follows that $m\le\frac{1}{2}(|G|-1)$. Equality can only hold if $G$ is a $3$-group. Then $p=2$ according to \cite[Theorem~1.II]{Fleischmann} and we easily get a contradiction. Hence, $m\le\frac{1}{2}(|G|-3)$. If we have equality this time, $G$ consists of elements of order $3$ and just one subgroup of order $5$. But then $G$ is nilpotent and must contain elements of composite order as well. This shows $m\le\frac{1}{2}(|G|-5)$. Since the number of subgroups of prime order $r$ is always congruent to $1$ modulo $r$ (Frobenius' extension of Sylow's Theorem), equality in this case leads to $G\cong C_7\rtimes C_3$. Of course, this is impossible and we conclude that $\frac{1}{2}(|G|-k)\le p\le m-1\le\frac{1}{2}(|G|-9)$ and $k=9$. However, it is not hard to see that there is no group with the desired number of minimal subgroups. 

\textbf{Case 2:} $k\in\{4,8\}$.\\
Now $G$ has even order and $p>2$. Hence, $|G|$ is not divisible by $4$. It is well-known that in this situation $G$ has a normal $2$-complement $N$. As usual, the number of minimal subgroups of $G$ inside $N$ is at most $\frac{1}{2}(|N|-1)<\frac{1}{4}|G|$. Let $S$ be a Sylow $2$-subgroup of $G$. Then the number of minimal subgroups of $G$ outside $N$ is $|G:\N_G(S)|$. 
If $S<\N_G(S)$, then $m<\frac{1}{4}|G|+\frac{1}{6}|G|=\frac{5}{12}|G|$.
Since $\frac{1}{2}|G|-4\le p\le m-1$, we derive that $|G|<36$. These cases have been handled in \cite[Proposition~14.9]{habil}.
Thus, let $S=\N_G(S)$. Then $S$ acts as inversion on $N$ and $N$ must be abelian. Ito's Theorem implies that $\dim P=2$ (see \cite[Theorem~6.15]{Isaacs}). 
Since $N\le G'\le\SL(2,q)$, $N$ is cyclic and so $G\in\{D_{2p+4},D_{2p+8}\}$. Both cases are excluded by \cite[Proposition~14.8]{habil}.

\textbf{Case~3:} $k=6$.\\
As before, $\frac{1}{2}|G|-3=p\le m-1$. Suppose first that $m=\frac{1}{2}|G|-2$. 
Let $\mathcal{M}=\{M_1,\ldots,M_{p+1}\}$ and $\lvert\C_P(M_i)\rvert=p^{d_i}$ with $d_1\ge\ldots\ge d_{p+1}$. 
Setting $|P|=p^n$ we refine \eqref{sum} to
\begin{equation}\label{union2}
|P|\le\sum_{i=1}^{p+1}\lvert\C_P(M_i)\rvert-\sum_{i=2}^{p+1}\lvert\C_P(M_1)\cap\C_P(M_i)\rvert\le p^{d_1}+(1-p^{d_1-n})\sum_{i=2}^{p+1}p^{d_i}\le p^n.
\end{equation}
This implies $d_1=\ldots=d_{p+1}=n-1$. 
In particular, every $M_i$ acts faithfully on $[M_i,P]\cong C_p$ and we deduce that $|M_i|\mid p-1$. Hence, $|M_i|\mid |G|-2(p-1)=8$ and $G$ is a $2$-group.
But then we may choose $M_i\le\Z(G)$ and obtain the contradiction $\C_P(M_i)=1$, since $\C_P(M_i)$ is $G$-invariant and $P$ is irreducible. This argument implies $m\ge\frac{1}{2}|G|-1$. 
If $G$ is still a $2$-group, $p$ must be a Fermat prime or a Mersenne prime by \cite[Theorem~2.II]{Fleischmann}. It follows easily that $p=5$ (a Fermat prime) and $|G|=16$. Here one can show with GAP~\cite{GAP48} that $G\cong D_8*C_4$ as given in the statement of our theorem.
Now suppose that $G$ is not a $2$-group. If $m=\frac{1}{2}|G|-1$, then, by \eqref{union2}, $G$ contains just one minimal subgroup of odd (prime) order. Since the number of involutions is always odd, we conclude that $|G|\equiv 2\pmod{4}$. But this gives the contradiction $p=2$. Hence, we may assume that $m\ge\frac{1}{2}|G|$. At this point we refer to the next case.

\textbf{Case~4:} $m\ge\frac{1}{2}|G|$.\\
First observe that this case includes the remaining possibility $k\le 2$, since $\frac{1}{2}(|G|-k)\le p\le m-1$. 
Moreover, $G$ is given as in \autoref{prop}. 
Since $\Z(G)$ is cyclic, some cases can be excluded immediately. If $G$ has odd order, then $m\le\frac{1}{2}(|G|-1)$ which is impossible. Therefore, $|G|$ is even and $p>2$. 

Suppose first that $G\cong A\rtimes C_2$ where $A$ is abelian and $C_2$ acts as inversion. 
Let $M$ be a minimal subgroup of $G$ lying inside $A$. Then $M\unlhd G$ and $\C_P(M)$ is $G$-invariant. It follows that $\C_P(M)=1$. Hence, in \eqref{union} we only need to consider the minimal subgroups outside $A$. 
This leads to $k\in\{2,6\}$ in \eqref{sum} after taking the last cases into account. By Ito's Theorem we have $\dim P=2$ and $\pcore_{2'}(A)\le\SL(2,q)$ is cyclic. Since also $\Z(G)$ is cyclic, $A$ contains at most one involution. This shows that $A$ is cyclic and $G\in\{D_{2p+2},D_{2p+6}\}$. The first case corresponds to the exception given in the statement of the theorem and the second case is excluded by \cite[Proposition~14.8]{habil}.

Next let $G\cong C_2^n\rtimes C_2$ with $n\ge 2$. Then $\lvert\Z(G)\rvert=2$ and \cite[Lemma~1.1]{Berkovich1} gives $|G:G'|=4$. A theorem of Taussky leads to $G\cong D_8$ (see \cite[Proposition~1.6]{Berkovich1}).
But then $p=3$ and $G\cong D_{2p+2}$.
Now let $G\cong C_2^{2n}\rtimes C_3$. Then Ito's Theorem implies $\dim P\le 3$ and the $2$-rank of $G$ is at most $3$ (see for example \cite[Proposition~7.13]{habil}). Consequently, $n=1$, $G\cong A_4$ and $p=5$. One can show that $A_4$ always has a regular orbit (see \cite[Proposition~14.9]{habil}).

Now we discuss the extraspecial group $G\cong D_8*\ldots*D_8$ with $n\ge 2$ factors (note that there are no elementary abelian direct summands, because $\Z(G)$ is cyclic). Then $p\ge\frac{1}{2}|G|-3\ge 13$ by the last cases. 
Let $z\in\Z(G)$ be the unique central involution. Then $\C_P(z)=1$, since $\C_P(z)$ is $G$-invariant. Hence, $z$ inverts the elements of $P$. The non-central involutions in $G$ can be paired in the form $\{x,xz\}$. Then $P=[x,P]\times\C_P(x)$ and $\C_P(x)\subseteq[xz,P]$. This yields $\lvert\C_P(x)\rvert\lvert\C_P(xz)\rvert\le|P|$ and \[\lvert\C_P(x)\rvert+\lvert\C_P(xz)\rvert\le\lvert\C_P(x)\rvert+\frac{|P|}{\lvert\C_P(x)\rvert}\le\frac{|P|}{p}+p.\] 
As in \eqref{sum}, we obtain 
\[|P|<\frac{m-1}{2}\Bigl(\frac{|P|}{p}+p\Bigr).\]
From \cite[Table~1]{BurnessStuart} it follows that $m\le\frac{3}{4}|G|$. 
In the case $|P|\ge p^3$ we derive the contradiction
\[p<\frac{m}{2}\Bigl(1+\frac{1}{p}\Bigr)\le\frac{3}{8}|G|\frac{14}{13}=\frac{21}{52}|G|\le\frac{1}{2}|G|-3.\]
Therefore, $|P|=p^2$ and $\dim P=2$. Then again $G$ has $2$-rank at most $3$ which contradicts $n\ge 2$.

In the case $\exp(G)=3$, $G$ has odd order which was already excluded.
Finally, the cases $G\in\{S_3\times D_8,S_3^2,S_4,A_5\}$ give $\frac{1}{2}|G|-3\le p\le m-1\le\frac{1}{2}|G|$ (see \cite[Table~1]{BurnessStuart}) and $k=2$. Again one can show that these groups have regular orbits (see \cite[Proposition~14.9]{habil}).
\end{proof}

\autoref{main} does not extend to $|G|=2p+10$. For example the semidihedral group $SD_{16}$ of order $16$ acts faithfully on $C_3^2$ without having regular orbits, since $16>9$. Similarly, $\texttt{SmallGroup}(24,8)$ is an automorphism group on $C_7^2$ without regular orbits. 
There is no doubt that one can classify these examples as well, but this becomes increasingly tedious and probably will not reveal new insights. 

One can check that the exception $G=D_8*C_4$ in \autoref{main} actually occurs. 
In the following we show that the other exceptions occur for every odd prime $p$.

\begin{Prop}\label{converse}
Let $p$ be an odd prime. Then there exists an automorphism group $G$ on a finite $p$-group $P$ such that $|G|=2p+2$ and $G$ has no regular orbit on $P$.
\end{Prop}
\begin{proof}
Let $P$ be a $2$-dimensional vector space over $\mathbb{F}_p$. We regard $P$ as the additive group of the field $\mathbb{F}_{p^2}$. Let $\gamma$ be a generator of $\mathbb{F}_{p^2}^\times$. Then $H:=\langle\gamma^{p-1}\rangle$ has order $p+1$ and acts by multiplication on $P$. Let $F$ be the Frobenius automorphism of $\mathbb{F}_{p^2}$. Then $F(\gamma^{p-1})=\gamma^{p^2-p}=\gamma^{1-p}$ and $G:=H\rtimes\langle F\rangle\cong D_{2p+2}$ acts faithfully on $P$. It suffices to show that $F$ fixes every orbit of $H$ on $P$. Any non-trivial element of $P$ has the form $\gamma^i$ for some $i\in\mathbb{Z}$. Then $F(\gamma^i)=\gamma^{ip}=\gamma^i\gamma^{i(p-1)}$. Hence, $\gamma^i$ and $F(\gamma^i)$ lie in the same orbit of $H$.
\end{proof}

For $p=2$ the smallest $p'$-automorphism group without regular orbits is the semilinear group $C_7\rtimes C_3$ acting on $C_2^3$. It is obvious from the group order that in this situation there are no regular orbits.

\section{Application}

In the next theorem we consider blocks of finite groups with respect to an algebraically closed field of characteristic $p>0$. We use the standard notation which can be found for example in \cite{habil}.

\begin{Thm}
Let $B$ be a $p$-block of a finite group with abelian defect group $D$ and inertial index $e\le 6p+5$. Then Brauer's $k(B)$-Conjecture holds for $B$, i.\,e. $k(B)\le|D|$.
\end{Thm}
\begin{proof}
By \cite[Theorem~14.13]{habil}, we may assume that $e>255$ and therefore $p\ge 43$.
The inertial quotient $I(B)$ of $B$ is a $p'$-group of order $e$ and acts faithfully on $D$. 
In order to find a large orbit of $I(B)$ on $D$, we may assume by \cite[Lemma~2.6.2]{HartleyTurull} that $D$ is elementary abelian of rank $n$. Then $I(B)\le\GL(n,p)$. Let $N:=I(B)\cap\SL(n,p)\unlhd I(B)$. We use the arguments from the proof of \autoref{main} to show that $N$ has a regular orbit on $D$. Suppose by way of contradiction that $N$ has no regular orbit. Let $\mathcal{M}$ be the set of minimal subgroups of $N$. Then $D=\bigcup_{M\in\mathcal{M}}{\C_D(M)}$. Since $N\le\SL(n,p)$, we have $\lvert[D,M]\rvert\ge p^2$ and $|D:\C_D(M)|\ge p^2$ for $M\in\mathcal{M}$. Consequently,
\[|D|<\sum_{M\in\mathcal{M}}\lvert\C_D(M)\rvert\le\lvert\mathcal{M}\rvert\frac{|D|}{p^2}\]
and $p^2<\lvert\mathcal{M}\rvert\le e\le 6p+5$. This contradicts $p\ge 43$. 

Hence, $N$ has a regular orbit, i.\,e. there exists $x\in D$ such that $\C_N(x)=1$. It follows that $\lvert\C_{I(B)}(x)\rvert\le|I(B)/N|$. Now if $|I(B)/N|$ is a prime or $1$, then the claim follows from \cite[Proposition~11]{SambaleC4}. In the remaining case there exists a normal subgroup $K\unlhd I(B)$ such that $|I(B)/K|$ is $4$ or an odd prime. We conclude that $|K|\le e/3<2p+2$. \autoref{main} implies that $K$ has a regular orbit on $D$. Now the claim follows from \cite[Lemma~14.5]{habil} and \cite[Corollary~12]{SambaleC4}.
\end{proof}

\section*{Acknowledgment}
I thank Julian Brough for making useful comments on a draft of this paper.
This work is supported by the German Research Foundation (project SA 2864/1-1) and the Daimler and Benz Foundation (project 32-08/13).


\begin{thebibliography}{10}

\bibitem{Berkovich1}
Y. Berkovich, \textit{Groups of prime power order. {V}ol. 1}, de Gruyter
  Expositions in Mathematics, Vol. 46, Walter de Gruyter GmbH \& Co. KG,
  Berlin, 2008.

\bibitem{BurnessStuart}
T.~C. Burness and S.~D. Scott, \textit{On the number of prime order subgroups
  of finite groups}, J. Aust. Math. Soc. \textbf{87} (2009), 329--357.

\bibitem{Espuelas}
A. Espuelas, \textit{The existence of regular orbits}, J. Algebra \textbf{127}
  (1989), 259--268.

\bibitem{Fleischmann}
P. Fleischmann, \textit{Finite groups with regular orbits on vector spaces}, J.
  Algebra \textbf{103} (1986), 211--215.

\bibitem{GAP48}
The GAP~Group, \textit{GAP -- Groups, Algorithms, and Programming, Version
  4.8.6}; 2016, (\url{http://www.gap-system.org}).

\bibitem{Hargraves}
B.~B. Hargraves, \textit{The existence of regular orbits for nilpotent groups},
  J. Algebra \textbf{72} (1981), 54--100.

\bibitem{HartleyTurull}
B. Hartley and A. Turull, \textit{On characters of coprime operator groups and
  the {G}lauberman character correspondence}, J. Reine Angew. Math.
  \textbf{451} (1994), 175--219.

\bibitem{Isaacs}
I.~M. Isaacs, \textit{Character theory of finite groups}, AMS Chelsea
  Publishing, Providence, RI, 2006.

\bibitem{KoehlerPahlings}
C. K{\"o}hler and H. Pahlings, \textit{Regular orbits and the
  {$k(GV)$}-problem}, in: Groups and computation, {III} ({C}olumbus, {OH},
  1999), 209--228, Ohio State Univ. Math. Res. Inst. Publ., Vol. 8, de Gruyter,
  Berlin, 2001.

\bibitem{Robinson}
G.~R. Robinson, \textit{On {B}rauer's {$k(B)$} problem}, J. Algebra
  \textbf{147} (1992), 450--455.

\bibitem{habil}
B. Sambale, \textit{Blocks of finite groups and their invariants}, Springer
  Lecture Notes in Math., Vol. 2127, Springer-Verlag, Cham, 2014.

\bibitem{Sbrauerfeit}
B. Sambale, \textit{On the {B}rauer-{F}eit bound for abelian defect groups},
  Math. Z. \textbf{276} (2014), 785--797.

\bibitem{SambaleC4}
B. Sambale, \textit{Cartan matrices and {B}rauer's {$k(B)$}-{C}onjecture {IV}},
  J. Math. Soc. Japan (to appear).

\bibitem{SambaleRank3}
B. Sambale, \textit{On blocks with abelian defect groups of small rank},
  Results Math. \textbf{71} (2017), 411--422.

\bibitem{kGVproblem}
P. Schmid, \textit{The solution of the {$k(GV)$} problem}, ICP Advanced Texts
  in Mathematics, Vol. 4, Imperial College Press, London, 2007.

\bibitem{RegOrbits}
Y. Yang, \textit{Regular orbits of nilpotent subgroups of solvable linear
  groups}, J. Algebra \textbf{325} (2011), 56--69.

\end{thebibliography}
\end{document}